\documentclass[11pt]{amsart}

\usepackage{amsmath, amssymb, amsthm}
\usepackage{mathtools}
\usepackage{hyperref}
\usepackage{todonotes}

\newtheorem{thm}{Theorem}[section]
\newtheorem{Def}[subsection]{Definition}
\newtheorem{ex}[subsection]{Example}
\newtheorem{Con}[subsection]{Conjecture}
\newtheorem{rem}[subsection]{Remark}
\newtheorem{prop}[subsection]{Proposition}

\newtheorem{lem}[subsection]{Lemma}

\newcommand{\cl}{\operatorname{cl}}
\renewcommand{\int}{\operatorname{int}}

\title[Cubical homology theories for hypergraphs]
{A cubical approach to homology theories for hypergraphs}

\author{Syed Hadi Ali Zaidi}
\address{Department of Mathematical Sciences, Institute of Business Administration (IBA), Karachi, Pakistan}
\email{shzaidi@iba.edu.pk}
\author{Samira Sahar Jamil}
\address{Department of Mathematics, University of Notre Dame, Notre Dame, Indiana, USA}
\email{sjamil@nd.edu}

\begin{document}

\begin{abstract}
We introduce three homology theories for hypergraphs, \linebreak
namely $\Gamma$-homology, $\Box$-homology, and $\times$-homology, and show that they are pairwise non-isomorphic and distinct from the embedded homology of hypergraphs. We further introduce a notion of homotopy for hypergraphs that extends the discrete homotopy theory of graphs. Among the homology theories considered, we prove that $\Box$-homology is invariant under this homotopy, whereas $\Gamma$-homology and $\times$-homology fail to satisfy homotopy invariance. Based on excision, we also identify a distinctive structural behavior exhibited by $\Gamma$-homology that further differentiates it from $\Box$-homology.
\end{abstract}

\maketitle


\section{Introduction}
Homology theories for discrete structures have attracted significant attention in recent years, motivated both by their connections with classical algebraic topology and by applications in combinatorics, category theory, and topological data analysis. Notable examples include discrete cubical homology for graphs \cite{HG_Barcelo19}, path homology for directed graphs \cite{HG_Grig18}, and homology theories for closure spaces \cite{HG_bubenik21, HG_bubenik22}.

Hypergraphs provide a natural generalisation of graphs and a flexible framework for modelling higher-order interactions; see \cite{HG_Emilie,HG_Bretto}. From an applied perspective, hypergraph homology has emerged as a tool for extracting topological features from complex datasets, where hyperedges encode multi-way interactions beyond pairwise relationships. Such methods have found applications in network science and data analysis \cite{HG_Emilie,HG_Bressan}, where homological invariants can capture higher-dimensional \linebreak structure not visible to classical graph-theoretic approaches; see also \cite{HG_bretto13,HG_Gaspar26,HG_Liu24} for further applications.

The embedded homology theory of hypergraphs was introduced in \linebreak \cite{HG_Bressan} and is defined from a chain complex naturally associated to a hypergraph $H$. Let $(C_\bullet(H),\partial_\bullet)$ denote the simplicial chain complex of the smallest simplicial complex containing $H$, and for each $n\geq 0$, let $D_n\subseteq C_n(H)$ be the subgroup generated by the $n$-hyperedges of $H$. Since $(D_n,\partial_n)$ does not, in general, form a chain complex, one defines the infimum chain complex $(\mathrm{Inf}_\bullet(H),\partial_\bullet)$ by
\[
\mathrm{Inf}_n(H)=D_n\cap \partial_n^{-1}(D_{n-1}),
\]
for each $n\geq 1$, and the resulting homology groups are called the embedded homology groups of $H$; see \cite{HG_Bressan}.

Since its introduction, embedded homology has been studied from several perspectives and has developed into an active area of research, both for its intrinsic homological properties and for its applications \cite{HG_Gaspar26,HG_Liu24,HG_Ren26,HG_Valentiner26}. Beyond embedded homology, a variety of alternative homological constructions for hypergraphs have also been proposed; see the survey \cite{HG_Gaspar26}.

In this paper, we introduce three homology theories for hypergraphs, namely $\Gamma$-homology, $\Box$-homology, and $\times$-homology, motivated by the cubical singular homology in algebraic topology. These theories are designed to detect classes of cycles that are different from those detected by embedded homology. In particular, $\Gamma$-homology detects both uniform and non-uniform cycles; see Example \ref{general 3-cycle}. The theories $\Box$-homology and $\times$-homology extend discrete cubical homology and $\times$-homology for graphs, respectively, whereas $\Gamma$-homology exhibits substantially different behaviour.

We show that these homology theories are, in general, pairwise non-isomorphic and capture distinct structural features of hypergraphs; see Examples \ref{eg gamma 3-cycle}-\ref{eg box and x 1-hom}. In particular, while the first embedded homology group coincides with $\Box$-homology for uniform hypergraphs, the two theories diverge in graphs containing non-uniform cycles. For instance, the first embedded homology group of the hypergraph in Example \ref{general 3-cycle} is trivial, whereas $H_1^\Gamma$ is nontrivial.

 Homotopy theory for hypergraphs has previously been studied in \cite{HG_Chen01}, where the notion of Graham homotopy was introduced via Graham reductions. We define a notion of homotopy for hypergraphs that extends the discrete homotopy theory of graphs. We prove that $\Box$-homology is invariant under this homotopy, whereas $\Gamma$-homology and $\times$-homology fail to satisfy homotopy invariance; see Example \ref{eg not homotopy invariant}. Determining a homotopy theory that is naturally associated with $\Gamma$-homology remains an interesting open problem.

Excision and the Mayer-Vietoris sequence are among the fundamental structural properties of classical homology theories. In this paper, we establish analogues of excision for each of the three homology theories introduced above. In particular, $\times$-homology satisfies a hypergraph analogue of the classical excision property, whereas $\Gamma$-homology and $\Box$-homology exhibit more subtle behaviour. For the latter two theories, we obtain only partial Mayer--Vietoris type sequences, highlighting their different behavior than embedded homology, which has a complete Mayer--Vietoris sequence as established in \cite{HG_Bressan}. The absence of complete Mayer--Vietoris sequences makes $\Gamma$-homology and $\Box$-homology significantly more difficult to compute, but also suggests richer categorical and homotopical structures underlying these theories. Similar phenomena have already appeared in the study of discrete cubical homology and graph homotopy theories; see \cite{HG_Kapulkin25}.

This paper is organised as follows. In Section \ref{Sec_prelim}, we recall the preliminary notions and constructions used throughout the paper. Section \ref{ch.hom} introduces the three homology theories and studies their basic properties, including analogues of the dimension, connectivity, and additivity axioms. In Section \ref{hmtpy}, we develop a homotopy theory for hypergraphs and prove the homotopy invariance of $\Box$-homology. Finally, Section \ref{sec:excision} establishes excision-like results for the homology theories under consideration, showing in particular that $\times$-homology satisfies a hypergraph analogue of excision, while $\Gamma$-homology and $\Box$-homology satisfy weaker excision-like properties.

\section{Preliminaries}\label{Sec_prelim}
We begin by recalling basic notions on hypergraphs. A hypergraph $G$ consists of a finite vertex set $V$ together with a collection $E \subseteq \mathcal{P}(V)$ of subsets of $V$, called hyperedges. The degree of a vertex is the number of hyperedges that contain it, and the size of a hyperedge $e$ is its cardinality. A hypergraph is called $k$-uniform if every hyperedge has size $k$; in particular, a $2$-uniform hypergraph is a graph.

Following \cite{HG_Emilie}, we recall the notion of $s$-paths. Let $s$ be a positive integer. An $s$-path of length $k$ between hyperedges $f$ and $g$ is a sequence of hyperedges
\[
f = e_{i_1}, e_{i_2}, \ldots, e_{i_k} = g
\]
such that for each $j = 2, \ldots, k$, one has $|e_{i_{j-1}} \cap e_{i_j}| \geq s$, and $e_{i_{j-1}} \not\subseteq e_{i_j}$. Furthermore, an $s$-path is said to be $k$-cycle if, in the above sequence $e_{i_1}=e_{i_{k+1}}$. A hypergraph is called $s$-connected if any two hyperedges can be joined by an $s$-path. 

A hypergraph morphism $f \colon G=(V,E) \to G'=(V',E')$ is a map $f \colon V \to V'$ such that for every hyperedge $e \in E$, there exists a hyperedge $e' \in E'$ with $f(e) \subseteq e'$. The image of $f$ is the hypergraph whose vertex set is $f(V)$ and whose hyperedges are the subsets $f(e)$ for $e \in E$.

Given a hypergraph $G=(V,E)$ with $V=\{v_1,\dots,v_n\}$ and $E=\{e_1,\dots,\\ e_m\}$, its dual hypergraph $G^*$ is defined by taking $E$ as the vertex set, and defining hyperedges $v_j^* = \{ e_i \mid v_j \in e_i \}$ for $j=1,\dots,n$. Thus, the roles of vertices and hyperedges are interchanged. We will later consider the relationship between a hypergraph and its dual from the perspective of homotopy theory.

Let $G=(V,E)$ be a hypergraph. An induced subhypergraph on a subset $V' \subseteq V$ is the hypergraph $G'=(V', E')$ with $E'=\{ e \cap V' \mid e \in E \}$ as the hyperedge set. More generally, a hypergraph $G'=(V',E')$ is called a subhypergraph of $G$ if $V' \subseteq V$ and each $e' \in E'$ is contained in some $e \in E$. Every induced subhypergraph is a subhypergraph, but the converse need not hold.

Given hypergraphs $G_1=(V_1,E_1)$ and $G_2=(V_2,E_2)$, their Cartesian product is the hypergraph $G_1 \times G_2$ with vertex set $V_1 \times V_2$ and hyperedges
\[
E_1 \times E_2 = \{ \{x\} \times e : x \in V_1,\, e \in E_2 \} \cup \{ e \times \{y\} : e \in E_1,\, y \in V_2 \}.
\]

Finally, two hypergraphs $G=(V,E)$ and $G'=(V',E')$ are said to be isomorphic if 
there exist bijection $f \colon V \to V'$  such that for all $e \in E$, $f(e)=e'$  and $f|_e$ is a bijection.

\section{Homology Theories for Hypergraphs}\label{ch.hom}

In this section we construct three homology theories for hypergraphs, namely $\Gamma$-homology, $\Box$-homology, and $\times$-homology. Since a $2$-uniform hypergraph is a graph, these constructions extend existing homology theories for graphs. In particular, we exhibit examples showing that $\Gamma$-homology is not isomorphic to discrete cubical homology.

\subsection{$\Gamma$-homology and cubical singular homology}
For each $n \geq 1$, let $V_n = \{0,1\}^n$. We define three hypergraph structures on $V_n$.

The $\Gamma$-hypercube has vertex set $V_n$ and hyperedges
\[
e_a = \{ v \in V_n \mid |v-a| \leq 1 \}, \quad a \in V_n,
\]
where $|v-a| = \sum_{i=1}^n |v_i - a_i|$.

The $\times$-hypercube is the complete graph on the vertex set $V_n$.

The $\Box$-hypercube is defined inductively as the $n$-fold Cartesian product of the $1$-dimensional hypergraph with vertex set $\{0,1\}$ and a single hyperedge $\{0,1\}$.

Let $J_n^*=(V_n^*,E_n^*)$ denote any of these hypercubes, where $* \in \{\Gamma, \Box, \times\}$.
A singular $(n,*)$-hypercube in a hypergraph $G$ is a hypergraph morphism
\[
\sigma \colon J_n^* \to G.
\]
For each $n \geq 0$, let $Q_n^*(G)$ be the free abelian group generated by all singular $(n,*)$-hypercubes.

A singular $(n,*)$-hypercube $\sigma$ is called degenerate if it is independent of one coordinate. In-particular, $\sigma$ is degenerate if for some $i,~\sigma(v_1,v_2,\ldots,v_{i-1},\\ 0,v_{i+1},\ldots,v_n)=\sigma(v_1,v_2,\ldots,v_{i-1},1,v_{i+1},\ldots,v_n)$. Let $D_n^*(G)$ be the subgroup generated by degenerate cubes, and define
\[
C_n^*(G) := Q_n^*(G)/D_n^*(G).
\]

For $\sigma \in Q_n^*(G)$, define face maps $A_i\sigma, B_i\sigma \in Q_{n-1}^*(G)$ by
\[
A_i\sigma(v_1,\dots,v_{n-1}) = \sigma(v_1,\dots,v_{i-1},0,v_i,\dots,v_{n-1}),
\]
\[
B_i\sigma(v_1,\dots,v_{n-1}) = \sigma(v_1,\dots,v_{i-1},1,v_i,\dots,v_{n-1}),
\]
for $1 \leq i \leq n$.

The boundary operator $\partial_n \colon Q_n^*(G) \to Q_{n-1}^*(G)$ is defined by
\[
\partial_n(\sigma) = \sum_{i=1}^n (-1)^i \big(A_i\sigma - B_i\sigma\big).
\]
This descends to a well-defined map $\partial_n \colon C_n^*(G) \to C_{n-1}^*(G)$, and one checks that $\partial_{n-1}\partial_n = 0$. Thus $(C_\bullet^*(G), \partial_\bullet)$ is a chain complex.

We define
\[
Z_n^*(G) := \ker \partial_n, \qquad B_n^*(G) := \operatorname{im} \partial_{n+1},
\]
and the homology groups
\[
H_n^*(G) := Z_n^*(G)/B_n^*(G).
\]

We now proceed to some examples and propositions which helps us in understanding the defined homologies.

\begin{ex}\label{eg single point hom} Consider an empty hypergraph $G$. Then 
$
Q_n^*(G)=D_n^*(G)=C_n^*(G)=\{0\} \quad \text{for all } n \geq 0.
$
If $G$ consists of a single vertex with a hyperedge containing it, then there is exactly one nondegenerate $(0,*)$-hypercube. Thus,
$
C_0^*(G)=Q_0^*(G)\cong \mathbb{Z},
$
while for all $n>0$, every $(n,*)$-hypercube is degenerate, so $C_n^*(G)=\{0\}$. Hence,
\[
H_n^*(G)=
\begin{cases}
\mathbb{Z} & n=0,\\
\{0\} & n \geq 1.
\end{cases}
\]
The same conclusion holds if $G$ has $n$ vertices with a single hyperedge containing all of them.
\\
Also, if $G$ consists of a single vertex with no hyperedge, then 
\[
C_n^*(G)=\{0\} \quad \text{for all } n \geq 0,
\]
and hence 
\[
H_n^*(G)=\{0\} \quad \text{for all } n \geq 0.
\]
\end{ex}

\begin{prop}\label{prop 1-connected}
	If a hypergraph $G$ is $1$-connected, then $H_0^*(G)\cong \mathbb{Z}$.
\end{prop}

\begin{proof}
Define $\epsilon:C_0^*(G)\to \mathbb{Z}$ by
\[
\epsilon\Big(\sum_i n_i \sigma_i\Big)=\sum_i n_i.
\]
Since $C_0^*(G)=Z_0^*(G)$, it suffices to show $\ker(\epsilon)=B_0^*(G)$.

For $\sum_i n_i \sigma_i \in C_1^*(G)$,
\[
\partial_1\Big(\sum_i n_i \sigma_i\Big)
= \sum_i n_i (B_1\sigma_i - A_1\sigma_i),
\]
so $\epsilon\circ \partial_1=0$, and hence $B_0^*(G)\subseteq \ker(\epsilon)$.

Conversely, let $\sum_i n_i v_i \in \ker(\epsilon)$, so $\sum_i n_i=0$. Fix $v\in V$. For each $v_i$, $1$-connectedness gives a sequence of hyperedges joining $v$ to $v_i$ through a sequence $v=v_{i_0},v_{i_1},\ldots,v_{i_k}=v_i$ such that for $j=1,\ldots,k-1$, each $v_{i_j}$ contained in the intersection of the hyperedges, yielding $(1,*)$-hypercubes of the form \[\sigma_{i_j}:J_1^*\to G\] with $\sigma_{i_j}(0)=v_{i_{j-1}}$ and $\sigma_{i_j}(1)=v_{i_j}$ such that 
\[
\partial_1\Big(\sum_j^k\sigma_{i_j}\Big)=v_{i_k} - v_{i_0}=v_i-v.
\]
Thus,
\[
\partial_1\Big(\sum_{i,j} n_i \sigma_{i_j}\Big)
= \sum_i n_i v_i - \Big(\sum_i n_i\Big)v
= \sum_i n_i v_i,
\]
so $\ker(\epsilon)\subseteq B_0^*(G)$.

Therefore, $\ker(\epsilon)=B_0^*(G)$ and
\[
H_0^*(G)=Z_0^*(G)/B_0^*(G)\cong \mathbb{Z}.
\]
\end{proof}

Note that, if hypergraph $G$ is $s$-connected, $H_0^*(G)=\mathbb{Z}$. The proof follows from $\ref{prop 1-connected}$.\\

\begin{prop}
	Let hypergragh $G=(V,E)$ and $G_1$ and $G_2$ be the non empty subhypergaphs such that $V=V_1 \bigsqcup V_2$ and $E= E(V_1) \bigsqcup E(V_2)$, where $V_1$ and $V_2$ are vertex sets of $G_1$ and $G_2$, respectively. Then $H_n^*(G)=H_n^*(G_1)\oplus H_n^*(G_2)$. 
\end{prop}
\begin{proof}
	We will just sketch the proof. For the case of $H_n^*(G)$ note that by definition of hypergraph morphism, $C_n^*(G)=C_n^*(G_1)\oplus C_n^*(G_2)$ and the result follows.
\end{proof}

%

\subsection{Homology and Cycles}

In the examples below, we use the following notation for $(n,*)$-hypercubes.
A $(1,*)$-hypercube $\sigma \colon J_1^* \to G$ will be written as
$\sigma_{ab}$, where $a,b\in V$ are the images of $0$ and $1$,
respectively.
Similarly, a singular $(2,*)$-hypercube $\sigma \colon J_2^* \to G$ will be written
as $\sigma_{abcd}$, where $a,b,c,d$ denote the images of
$(0,0),(1,0),(0,1),(1,1)$, respectively. The examples \ref{eg gamma 3-cycle},\ref{eg x and box hom},\ref{eg box and x 1-hom} show that all three of the defined homologies are distinct from each other. In particular, examples \ref{eg gamma 3-cycle} and \ref{eg x and box hom} show that the $\Gamma$-homology is distinct from the $\Box$-homology and the $\times$-homology. Whereas, examples \ref{eg x and box hom} and \ref{eg box and x 1-hom} show that the $\Box$-homology and $\times$-homology are distinct from each other.

\begin{ex}\label{eg gamma 3-cycle}
Let $G=(V,E)$ be the hypergraph with vertex set $V=\{a,b,c\}$ and hyperedge set $E=\{\{a,b\},\{a,c\},\{b,c\}\}$. Then
\[
C_0^\Gamma(G)=\langle a,b,c \rangle \cong \mathbb{Z}^3,
\]
and
\[
C_1^\Gamma(G)=\langle \sigma_{ab}, \sigma_{ba}, \sigma_{ac}, \sigma_{ca}, \sigma_{bc}, \sigma_{cb} \rangle,
\]
where $\sigma_{ab}(0)=a$ and $\sigma_{ab}(1)=b$.
\\
A direct computation (via the Smith normal form of $\partial_1$) shows that
\[
Z_1^\Gamma(G)\cong \mathbb{Z}^4, \quad B_0^\Gamma(G)\cong \mathbb{Z}^2,
\]
and hence $H_0^\Gamma(G)\cong \mathbb{Z}$.
\\
Similarly, boundaries of the $(2,\Gamma)$-cubes $\sigma_{abaa}, \sigma_{acaa}, \sigma_{bcbb}$ generate $\operatorname{Im}(\partial_2)\cong \mathbb{Z}^3$, yielding
\[
H_1^\Gamma(G)\cong \mathbb{Z}.
\]
\end{ex}

\begin{ex}\label{eg x and box hom} 
	Let $G=(V,E)$ be the hypergraph with vertex set $V=\{a,b,c\}$ and set of hyperedges $E=\{\{a,b\},\{b,c\},\{a,c\}\}$. We have \[{C}_0^\times(G)=\langle \sigma_a,\sigma_b,\sigma_c\rangle \cong \mathbb{Z}^3\] where $\sigma_a,\sigma_b,$ and $\sigma_c$ are  $(0,\times)$-hypercubes. 	
    \[{C}_1^\times(G)=\langle \sigma_{ab},\sigma_{ba},\sigma_{bc},\sigma_{cb},\sigma_{ca},\sigma_{ac}\rangle \cong \mathbb{Z}^6.\] Then after computing boundary matrix of $\partial_1$ we have  \[{Z}_1^\times(G)\cong\mathbb{Z}^4 ~\text{and} ~ {B}_0^\times(G)=\mathbb{Z}^2.\] Hence, \[H_0^\times(G)=\mathbb{Z}.\] There are 48 non-degenerate $(2,\times)$-hypercubes, and after computation we have, \[{B}_1(J_1^\times,G)\cong \mathbb{Z}^4, \] and thus, \[H_1^\times(G)\cong \{0\}.\] Similar computation with same hypergraph $G$ gives \[H_1^\Box(G)\cong \{0\}.\] 
\end{ex}

\begin{ex}\label{eg box and x 1-hom}
Let $G=(V,E)$ with $V=\{a,b,c,d\}$ and $E=\{\{a,b\}, \{a,d\},\\ \{b,c\}, \{c,d\}\}$ then \[{C}_0^\Box(G)={Z}_0^\Box(G)=\mathbb{Z}^4\] and \[{C}_1^\Box(G)=\langle \sigma_{ab}, \sigma_{ba},\sigma_{ad},\sigma_{da},\sigma_{bc},\sigma_{cb}, \sigma_{cd},\sigma_{dc}\rangle.\] After the computation of the boundary matrix $\partial_{1}({C}_1^\Box(G))$, we have \[{Z}_1^\Box(G)=\mathbb{Z}^5 \text{ and } {B}_0^\Box(G)=\mathbb{Z}^3.\] Thus, \[H_0^\Box(G)= \mathbb{Z}.\] The Rank of the boundary matrix $\partial_{2}({C}_2^\Box(G))$ is $5$. Hence, \[H_1^\Box(G)=\{0\}.\] Note that, \[{C}_0^\times(G)={Z}_0^\times(G)=\mathbb{Z}^4.\] Now, the computation of the boundary matrix of $\partial_1({C}_1^\times(G))$ gives, \[{Z}_1^\times(G)=\mathbb{Z}^5 \text{ and } {B}_0^\times(G)=\mathbb{Z}^3.\] Thus, \[H_0^\times(G)=\mathbb{Z}.\] Rank of the boundary matrix of $\partial_{2}({C}_2^\times(G))$ is 4. Hence \[H_1^\times(G)=\mathbb{Z}.\]
\end{ex}

If there exists hypergraph morphism $f:G\rightarrow H$ and  $(n,*)$-hypercube $\sigma$ then one can define homomorphism $f_{\#}: Q_n^*(G)\rightarrow  Q_n^*(H)$ such that $f_{\#}(T)=f\circ \sigma$. Since $f$ and $\sigma$ are hypergraph morphism, then $f_{\#}(\sigma)$ is also a $(n,*)$-hypercube and $f_{\#}$ maps $ D_n^*(G)$ to $ D_n^*(H)$. Therefore, $f_{\#}: C_n^*(G)\rightarrow  C_n^*(H)$.
Note that $f_{\#} \partial_n=\partial_{n}f_{\#}$. Hence we have $f_{*}: H_m^*(G) \rightarrow  H_m^*(H)$. Therefore, $H_n^*$ is a functor.

We next exhibit a family of hypergraphs for which
$H_1^\Gamma$ can have arbitrarily large rank. In particular, for
$n\geq2$, let $G_n^m$ be an $m$-uniform hypergraph consisting of $n$
$3$-cycles arranged consecutively, with each pair of consecutive
$3$-cycles sharing exactly one hyperedge (see Figure \ref{fig:Gnm}). Precisely, 

\begin{Def}\label{def_Gmn}
For integers $m\ge2$ and $n\ge1$, define the $m$-uniform hypergraph $G_n^m$ recursively as follows. Let $G_1^m$ consist of three hyperedges
\[
e_1,e_2,e_3
\]
such that
\[
|e_1\cap e_2|
=
|e_2\cap e_3|
=
|e_3\cap e_1|
=1,
\]
with the three pairwise intersections distinct. For $n\ge2$, obtain $G_n^m$ from $G_{n-1}^m$ by choosing an end hyperedge $e_{2n-1}$ and adjoining two new $m$-element hyperedges $e_{2n}$ and $e_{2n+1}$ such that
\[
|e_{2n-1}\cap e_{2n}|
=
|e_{2n}\cap e_{2n+1}|
=
|e_{2n+1}\cap e_{2n-1}|
=1,
\]
where these three intersection vertices are distinct, and such that
$e_{2n}$ and $e_{2n+1}$ have no other vertices in common with $G_{n-1}^m$.
\end{Def}
\begin{figure}[h]
    \centering
    \includegraphics[width=0.5\linewidth]{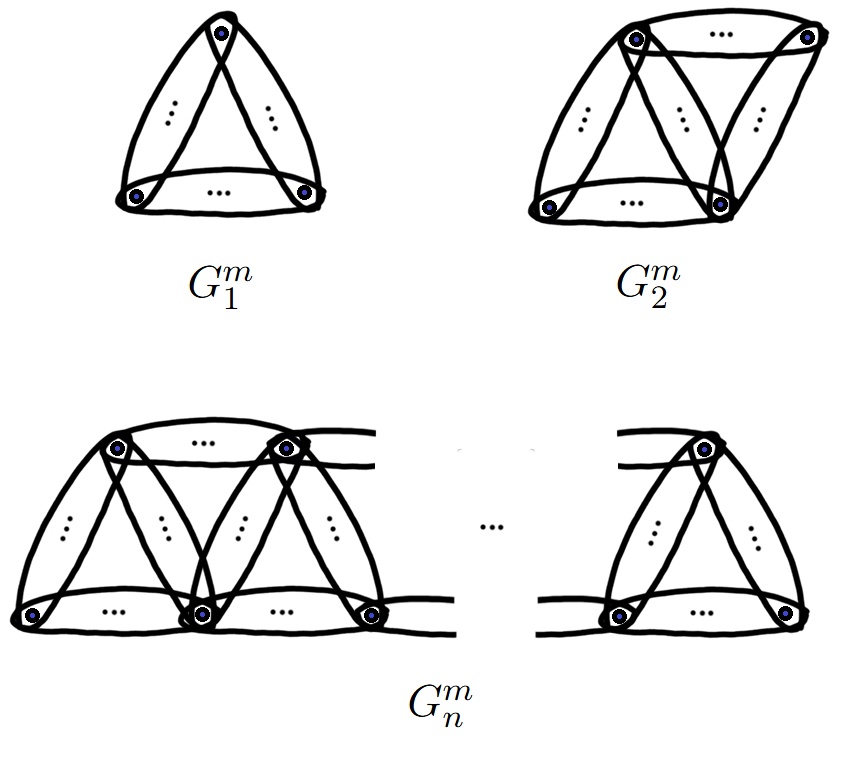}
    \caption{The hypergraph $G_n^m$.}
    \label{fig:Gnm}
\end{figure}

\begin{prop}\label{prop_h1_Gmn}
Let $m=2$ and $n\ge1$. Then
\[
H_1^\Gamma(G_n^2)\cong \mathbb{Z}^n.
\]
\end{prop}




    

\begin{proof}   We first note that the hypergraph $G_n^2$ has    $2n+1$ hyperedges, it has $n+2$ vertices. Since $G_{n}^2$ has $n+2$ vertices, 
    \[C_0^\Gamma(G_{n}^2)=Z_0^\Gamma(G_{n}^2)\cong \mathbb{Z}^{n+2}. \] 
    Since $G_{n}^2$  is a 1-connected hypergraph we have $H_0^\Gamma(G_{n}^2)\cong\mathbb{Z}$, 
    therefore
            \[ B_0^\Gamma(G_{n}^2)\cong \mathbb{Z}^{n+1}.
           \]
           Furthermore, $G_{n}^2$ has $2n+1$  hyperedges. 
           Since each
hyperedge with vertices $a,b$ gives rise to two generators
$\sigma_{ab}$ and $\sigma_{ba}$ of $C_1^\Gamma(G_{n+1}^2)$, we have,
           \[ C_1^\Gamma(G_{n}^2)\cong \mathbb{Z}^{4n+2}, \] which implies \[Z_1^\Gamma(G_{n}^2)\cong \mathbb{Z}^{3n+1}.\] 
           For each hyperedge with vertices $a,b$, boundary of the singular $(2,\Gamma)$-hypercube $\sigma_{abaa}$ equals $\sigma_{ab}+\sigma_{ba}$. 
           Therefore, two generators of 
           $C_1^\Gamma(G_{n}^2)$
           contribute to one generator
           \[
           \sigma_{ab}+\sigma_{ba}\in B_1^
           \Gamma(G_{n}^2)\]
           Since the hyperedges are distinct, these
relations are independent,
           therefore the rank of
           $B_1^\Gamma(G_{n}^2)$ equals half the rank of $C_1^\Gamma(G_{n}^2)$. 
           We have \[  B_1^\Gamma(G_{n}^2)\cong \mathbb{Z}^{2n+1}.\]  
           Hence, \[H_1^\Gamma(G_{n}^2)=\frac{\mathbb{Z}^{3n+1}}{\mathbb{Z}^{2n+1}}\cong \mathbb{Z}^{n}.\]

           \end{proof}

The preceding result can be extended to $m$-uniform hypergraphs: for $m\ge3$ and $n\ge1$, we obtain
\[
H_1^\Gamma(G_n^m)\cong\mathbb{Z}^n.
\]

The example below shows that the first $\Gamma$-homology group of a hypergraph $G$ consisting of a non-uniform $4$-cycle is not trivial, whereas the first embedded homology group, $H_1(\text{inf}(G))$, of the hypergraph in example \ref{general 3-cycle}, is trivial. 

\begin{ex}[A non-uniform $4$-cycle]\label{general 3-cycle}
Let $G$ be the hypergraph with
\[
V=\{a,b,c,d,e\}, \quad E=\{\{a,b,c\},\{b,e\},\{c,d\},\{d,e\}\}.
\]
A direct computation shows
\[
H_1^\Gamma(G)\cong \mathbb{Z}.
\]
\end{ex}

\begin{rem}
The above example \ref{general 3-cycle} shows that $\Gamma$-homology captures not only uniform cycles but also non-uniform cycle structures. Furthermore,  examples \ref{eg gamma 3-cycle} and \ref{eg box and x 1-hom} show that the $\Gamma$-homology and the $\times$-homology, respectively, for graphs (2-uniform hypergraphs), detects $3$-cycles and $4$-cycles, unlike discrete cubical homology. 
\end{rem}

\section{Homotopy for Hypergraphs}\label{hmtpy}

In this section, we introduce the notion of homotopy between hypergraph morphisms, and prove that $\Box$-homology is invariant under homotopy. We also discuss structural properties of $\Gamma$-homology and present examples illustrating its behaviour.




\subsection{Homotopy of Hypergraph Morphisms}

\begin{Def}
Let $G=(V,E)$ and $G'=(V',E')$ be hypergraphs, and let $f,g:G\rightarrow G'$ be hypergraph morphisms. We say that $f$ is \emph{$m$-step homotopic} to $g$, denoted by $f\simeq_m g$, if there exists a hypergraph morphism
\[
F:G\times J_m \rightarrow G'
\]
such that
\begin{itemize}
	\item $F(v,0)=f(v)$ for all $v\in V$,
	\item $F(v,m)=g(v)$ for all $v\in V$,
\end{itemize}
where $J_m$ is the hypergraph with vertex set $\{0,1,\ldots,m\}$ and hyperedges $\{\{0,1\},\ldots,\{m-1,m\}\}$.
\end{Def}

\begin{Def}
Two hypergraphs $G$ and $H$ are said to be \emph{homotopy equivalent} if there exist hypergraph morphisms $f:G\rightarrow H$ and $g:H\rightarrow G$ such that
\[
f\circ g \simeq Id_H \quad \text{and} \quad g\circ f \simeq Id_G.
\]
\end{Def}

\begin{Con}\label{dual HE}
	Let $G$ be a hypergraph and $G^*$ its dual. Then $G\simeq G^*$.
\end{Con}

\subsection{Homotopy Invariance of $\Box$-Homology}

\begin{thm}\label{thm 2}
Let $f,g:G\rightarrow H$ be hypergraph morphisms. If $f\simeq_m g$, then
\[
f_*= g_*:H_n^\Box(G)\rightarrow H_n^\Box(H).
\]
\end{thm}

\begin{proof}
	Define $F_j(x,t)=F(x,j+t-1)$, where $j\in [1,m]_{\mathbb{Z}}$, using the homotopy $F$. Construct chain homotopies $\phi_n$ by
	\[
	\phi_n(\sigma)(x_1,\ldots,x_{n+1})=F_j(\sigma(x_1,\ldots,x_n),x_{n+1}).
	\]
	A direct computation shows
	\[
	g_{\#}-f_{\#}=\partial_{n+1}\phi_n+\phi_{n-1}\partial_n.
	\]
	Hence $f_*=g_*$.
\end{proof}

\begin{prop}\label{homotopy equivalent}
	If $G\simeq H$, then
	\[
	H_n^\Box(G)\cong H_n^\Box(H).
	\]
\end{prop}

\begin{proof}
	Follows immediately from Theorem \ref{thm 2}.
\end{proof}

\subsection{Examples}

\begin{ex}[Homotopy of hypergraph morphisms]
Let $G=(V,E)$ and $H=(V',E')$ be hypergraphs with
\[
V=\{a,b,c,d\}, \quad V'=\{u,w,x,y,z\},
\]
\[
E=\{\{a,b,c\},\{a,c\},\{c\},\{a,c,d\}\}, \quad
E'=\{\{x,y,z\},\{u,x,y\},\{u,x,w\}\}.
\]
Define morphisms $f,g:G\to H$ by
\[
f(a)=w,\; f(b)=x,\; f(c)=u,\; f(d)=x,
\]
\[
g(a)=y,\; g(b)=z,\; g(c)=x,\; g(d)=u.
\]
Then $f\simeq_5 g$ via the homotopy given in the following table.
	{\fontsize{8.5pt}{10pt}\selectfont
    \[ \begin{tabular}{|c|c|c|c|c|c|}
		\hline
		$F(X,0)$& $F(X,1)$ & $F(X,2)$ & $F(X,3)$ & $F(X,4)$ & $F(X,5)$ \\
		\hline
		$F(a,0)=f(a)=w$&  $F(a,1)=x$ & $F(a,2)=y$& $F(a,3)= y$  & $F(a,4)= y$ &  $F(a,5)= y=g(a)$\\
		\hline
		$F(b,0)=f(b)=x$&  $F(b,1)=x$ & $F(b,2)=x$ & $F(b,3)=x$ & $F(b,4)=x $  & $F(b,5)= z=g(b)$ \\
		\hline
		$F(c,0)=f(c)=u$ & $F(c,1)=u$ & $F(c,2)=u$ & $F(c,3)=x$ & $F(c,4)= x$ & $F(c,5)=x=g(c)$ \\
		\hline
		$F(d,0)=f(d)=x$& $F(d,1)=x$ & $F(d,2)=x$ & $F(d,3)$=x & $F(d,4)=u$ & $F(d,5)=u=g(d)$ \\
		\hline
	\end{tabular}
	\]}

\end{ex}

The next example shows that $\Gamma$-homology and $\times$-homology are not homotopy invariant.
\\
\begin{ex}\label{eg not homotopy invariant}
Let $G$ be the hypergraph with vertex set $V=\{a,b,c,d\}$ and hyperedge set $E=\{\{a,b\},\{b,c\},\{c,d\},\{a,d\}\}$, and let $G'$ be the hypergraph with $V'=\{e\}$ and $E'=\{\{e\}\}$. Define morphisms $f:G\to G'$ by $f(V)=e$ and $g:G'\to G$ by $g(e)=a$ such that $G\simeq G'$. By Example~\ref{eg box and x 1-hom}, \[H_1^\times(G)\cong \mathbb{Z} \text{ and } H_1^\Gamma(G)\cong \mathbb{Z},\] whereas by example~\ref{eg single point hom} we have,
\[H_1^\times(G')=\{0\} \text{ and } H_1^\Gamma(G')=\{0\}.\]
\end{ex}

\section{Excision}\label{sec:excision}

Excision is a fundamental property of homology theories in algebraic topology, allowing one to compute homology groups by removing suitable subspaces. In this section, we investigate whether the homology theories $H_n^*(G)$, for $* \in \{\Gamma, \times, \Box\}$, satisfy an excision-like property.

We begin by introducing the necessary definitions.

Let $G = (V,E)$ be a hypergraph, and let $G_1 = (V_1,E_1)$ and $G_2 = (V_2,E_2)$ be subhypergraphs of $G$. We define $G_1 \cup G_2$ and $G_1 \cap G_2$ to be the induced subhypergraphs of $G$ with vertex sets
\[
V(G_1 \cup G_2) = V_1 \cup V_2, \qquad V(G_1 \cap G_2) = V_1 \cap V_2.
\]
Similarly, we define $G - G_1$ to be the induced subhypergraph of $G$ with vertex set $V \setminus V_1$.

\medskip

\begin{Def}[{\cite{HG_Bretto}}]\label{def:neighborhood}
Let $G = (V,E)$ be a hypergraph and let $x \in V$. The \emph{neighborhood} of $x$ in $G$ is the induced subhypergraph $N(x,G)$ with vertex set
\[
V(N(x,G)) = \{ y \in V \mid \exists\, e \in E \text{ such that } \{x,y\} \subseteq e \}.
\]
\end{Def}

\medskip

Let $G' = (V',E')$ be a subhypergraph of $G$. The \emph{closure} of $G'$ in $G$ is the induced subhypergraph $\cl(G')$ with vertex set
\[
V(\cl(G')) = \{ v \in V \mid \exists\, e \in E \text{ such that } v \in e \text{ and } e \cap V' \neq \emptyset \}.
\]

For a vertex $x \in V$, we write $\cl(x)$ for $\cl(\{x\})$. Note that
\[
V(N(x,G)) = V(\cl(x))
\]

\medskip

The \emph{interior} of $G'$ in $G$ is the induced subhypergraph $\int(G')$ with vertex set
\[
V(\int(G')) = V \setminus V(\cl(G - G')).
\]

\medskip

The hypergraph $J_n^*$ can be partitioned into induced subhypergraphs $S_j$ for any non-negative integer $j$, defined as follows:

\[
S_j = \{ (v_1, v_2, \ldots, v_n) \mid \sum_{i=1}^n v_i = j \}
\]
Note that $S_j=\emptyset$, for $j\geq n+1$, and
$S_0$ and $S_n$ are singletons.

\medskip

Let $G'$ be an induced subhypergraph of $G$, and let $i : G' \hookrightarrow G$ denote the inclusion. Then the induced map
\[
i_\#: C_n^*(G') \to C_n^*(G)
\]
is injective. Thus we identify $C_n^*(G')$ with a subgroup of $C_n^*(G)$ and define
\[
C_n^*(G,G') := \frac{C_n^*(G)}{C_n^*(G')}.
\]

The boundary maps descend to give homomorphisms
\[
\partial_n^{(*,\mathrm{rel})} : C_n^*(G,G') \to C_{n-1}^*(G,G'),
\]
and we define
\[
Z_n^*(G,G') := \ker \partial_n^{(*,\mathrm{rel})}, \qquad
B_n^*(G,G') := \operatorname{im} \partial_{n+1}^{(*,\mathrm{rel})}.
\]
The \emph{relative homology groups} are
\[
H_n^*(G,G') := \frac{Z_n^*(G,G')}{B_n^*(G,G')}.
\]

\subsection{Excision}

The following proposition records basic properties of closure and interior analogous to those in topology.

\begin{prop}\label{prop:closure-interior}
Let $G=(V,E)$ be a hypergraph and let $A,B$ be subhypergraphs of $G$. Then:
\begin{enumerate}
    \item $A \subseteq \cl(A)$ and $\int(A) \subseteq A$.
    \item $\int(G - A) = G - \cl(A)$ and $G - \int(A) = \cl(G - A)$.
    \item If $A \subseteq B$, then $\cl(A) \subseteq \cl(B)$ and $\int(A) \subseteq \int(B)$.
    \item $G = \int(A) \cup \int(B)$ if and only if $\cl(G - B) \subseteq \int(A)$.
\end{enumerate}
\end{prop}

\begin{proof}
The first two statements follow directly from the definitions. For the identity
\[
G - \int(A) = \cl(G - A),
\]
note that
\[
\int(A) = G - \cl(G - A),
\]
and hence
\[
G - \int(A) = G - (G - \cl(G - A)) = \cl(G - A).
\]
The remaining statements follow immediately from the definitions.
\end{proof}

\medskip
The following result will be used in proving that excision holds for the $\times$-homology.

\begin{lem}\label{times hom lem}
Let $G=(V,E)$ be the hypergraph, $G_1$ and $G_2$ be any two subhypergraphs such that there is some positive integer $i$ with  $G=\int(G_1)\cup \int(G_2)$. Then for every $(n,\times)$-hyppercube $\sigma$, $Im(\sigma)\subset G_1$ or $Im(\sigma) \subset G_2.$ In particular, $C_n^\times(G)=C_n^\times( G,\mathbb{U})$, where $\mathbb{U}=\{G_1,G_2\}$.
\end{lem}

\begin{proof} Note that by definition, $C_n^\times(G,\mathbb{U}) \subset C_n^\times(G)$. For converse we proceed through induction,
For $n=0$, let $\sigma \in  C_0^\times(G)$. Clearly, either $Im(\sigma) \in \int(G_1)$ or $Im(\sigma) \in \int(G_2)$. It implies $C_0^\times(G)\subseteq C_0^\times(G,\mathbb{U})$. Let us consider for $n=k$, we have $C_n^\times(G)\subset C_n^\times(G,\mathbb{U})$.

Induction step: $n=k+1$, Assume without loss of generality that $\sigma(S_0) \in \int(G_1).$ Note that,  $\sigma(S_0 \cup S_1 \cup S_2 \cup\ldots \cup S_{k+1}) \subseteq V_1.$ Since all of the vertices of $S_i's$ are in neighbors of $\sigma(S_0)$ in $J^\times_{k+1}$. Thus, $Im(\sigma)\subset G_1$, as desired. 
\end{proof}

\begin{thm}
(Excision for $\times$-homology)\\
Let $G=(V,E)$ be hypergraph. Let $G_1=(V_1,E(V_1))$ and $G_2=(V_2,E(V_2))$ are subhypergraph such that $G_1\subseteq G_2$ and $\cl(G_1)\subseteq \int(G_2)$. Then the inclusion $(G-G_1,G_2-G_1)\rightarrow (G,G_2)$ induces isomorpohisms $H_{n}^{\times}(G-G_1,G_2-G_1)\cong H_{n}^{\times}(G,G_2)$ for all $n$. Equivalently, for interior cover $\{A,B\}$ of hypergraph $G=(V,E)$ the inclusion $(B,A\cap B)\rightarrow (G,A)$ induces isomorphisms $H_{n}^{\times}(B,A\cap B)\cong H_{n}^{\times}(G,A), \forall n$.
\end{thm}
		\begin{proof}
The result follows directly from Lemma \ref{times hom lem} by a standard argument analogous to the classical case.
\end{proof}
\medskip

Next, we study excision for $\Gamma$-homology. We show that the classical excision property fails in this setting, and therefore formulate a weaker excision-like property adapted to $\Gamma$-homology.
\begin{lem}\label{lem:Gamma-localization}
Let $G=(V,E)$ be a hypergraph and let $\mathbb{U}=\{G_1,G_2\}$ be a family of subhypergraphs such that
\[
G = \int(G_1) \cup \int(G_2).
\]
Then for $n \in \{0,1,2,3,4\}$,
\[
C_n^\Gamma(G) = C_n^\Gamma(G,\mathbb{U}).
\]
\end{lem}

\begin{proof}
By definition, $C_n^\Gamma(G,\mathbb{U}) \subseteq C_n^\Gamma(G)$. We prove the reverse inclusion.

Let $\sigma : J_n^\Gamma \to G$ be an $(n,\Gamma)$-hypercube. We show that the image of $\sigma$ is contained entirely in $G_1$ or $G_2$.

For $j \ge 0$, recall that
\[
S_j = \{(v_1,\dots,v_n) \in V(J_n^\Gamma) \mid \textstyle\sum_{i=1}^n v_i = j\}.
\]
Note that $S_0$ consists of a single vertex.

Since $G = \int(G_1) \cup \int(G_2)$, we may assume without loss of generality that
\[
\sigma(S_0) \subseteq V(\int(G_1)).
\]
By Definition~\ref{def:neighborhood}, all vertices in $\sigma(S_1)$ and $\sigma(S_2)$ are neighbors of $\sigma(S_0)$, and hence lie in $V(G_1)$. Therefore,
\[
\sigma(S_0 \cup S_1 \cup S_2) \subseteq V(G_1).
\]

If $n \le 2$, it follows immediately that $\operatorname{Im}(\sigma) \subseteq G_1$.

Now suppose $n \in \{3,4\}$. If
\[
\sigma(S_3 \cup S_4) \subseteq V(G_1),
\]
then $\operatorname{Im}(\sigma) \subseteq G_1$, as required.

Otherwise, there exists $x \in S_3 \cup S_4$ such that $\sigma(x) \notin V(G_1)$. Then $\sigma(x) \in V(G \setminus G_1)$, and by Proposition~\ref{prop:closure-interior},
\[
\sigma(S_2 \cup S_3 \cup S_4) \subseteq V(\cl(G - G_1)).
\]
Since $\cl(G - G_1) \subseteq \int(G_2)$, we obtain
\[
\sigma(S_2 \cup S_3 \cup S_4) \subseteq V(\int(G_2)).
\]
Using the definition of interior again, it follows that
\[
\sigma(S_0 \cup S_1) \subseteq V(G_2).
\]
Hence $\operatorname{Im}(\sigma) \subseteq G_2$.

This proves that every $(n,\Gamma)$-hypercube lies in $C_n^\Gamma(G,\mathbb{U})$, and therefore
\[
C_n^\Gamma(G) \subseteq C_n^\Gamma(G,\mathbb{U}).
\]
\end{proof}

Next we show that the above lemma fails for the case of  $(5,\Gamma)$-hypercube. In particular, we have a hypergraph $G$ and subhypergraphs $G_1$ and $G_2$, 
with $G=\int(G_1)\cup \int(G_2)$ and singular $(5,\Gamma)$-hypercube $\sigma:J_5^\Gamma\to G$, so that the image of $\sigma$ is not entirely contained in $G_1$ or $G_2$.
\begin{ex}
	
	Let $G$ be the hypergraph with vertex set $V=\{v_1,\ldots, v_{6}\}$ and set of hyperedges $E=\{\{v_1,v_2,v_3\}, \{v_2, v_3,v_4\},\{v_3,v_4,v_5\}, \{v_4,v_5,v_6\},\}.$ We take $G_1=\{v_1,\ldots, v_5\}$ and $G_2=\{v_2,\ldots, v_{6}\}$. Note that $\int(G_1)=\{v_1,v_2,v_3\}$ and $\int(G_2)=\{v_4,v_5,v_{6}\}$ and thus $G=\int(G_1)\cup \int(G_2)$. Now define $\sigma: J_5^\Gamma\rightarrow G$ by $\sigma(S_0)=v_1$, $\sigma(S_1)=v_2$, $\sigma(S_2)=v_3$, $\sigma(S_3)=v_4$, $\sigma(S_4)=v_5$ and $\sigma(S_5)=v_{6}$. Note that $\sigma $ is a hypergraph morhism and $Im(\sigma) \notin G_i$ for $i=\{1,2\}$. 
\end{ex}

\medskip

\begin{thm}\label{gamma excision}
	(Excision-like for $\Gamma$-homology)\\
	Let $G=(V,E)$ be hypergraph. Let $G_1=(V_1,E_1)$ and $G_2=(V_2,E_2)$ are  subhypergraphs such that $G_1\subseteq G_2$ and $\cl(G_1)\subseteq \int(G_2)$. Then the inclusion $(G-G_1,G_2-G_1)\rightarrow (G,G_2)$ induces isomorpohisms $H_{n}^\Gamma(G-G_1,G_2-G_1)\cong H_{n}^\Gamma(G,G_2)$ for all $n<4$. Equivalently, for interior cover $\{A,B\}$ of hypergraph $G=(V,E)$ the inclusion $(B,A\cap B)\rightarrow (G,A)$ induces isomorphisms $H_{n}^\Gamma(B,A\cap B)\cong H_{n}^\Gamma(G,A), \forall n<4$.
\end{thm}

Note that the equivalence of the two statements follows by setting $B = G-G_1$ and $G_2=G-B$. Then $A\cap B=G_2-G_1$. 

\begin{proof}\label{excision proof}
	We prove the second statement, involving an interior cover of hypergraph $G=(V,E)$, $\mathbb{U}=\{A,B\}$. By $\ref{lem:Gamma-localization}$ we have $C_{n}^\Gamma(G,\mathbb{U})=C_{n}^\Gamma(G)$. Note that $C_{n}^\Gamma(G,\mathbb{U})=C_{n}^\Gamma(A)+C_{n}^\Gamma(B)$ and $C_{n}^\Gamma(A\cap B)=C_{n}\Gamma(A)\cap C_{n}^\Gamma(B)$. Furthermore, the map \[\frac{C_{n}^\Gamma(B)}{C_{n}^\Gamma(A)\cap C_{n}^\Gamma(B)}\rightarrow \frac{C_{n}^\Gamma(A)+C_{n}^\Gamma(B)}{C_{n}^\Gamma(A)}\] induced by inclusion is an isomorphism. Thus,
	\[C_{n}^\Gamma(B,A\cap B)=\frac{C_{n}^\Gamma(B)}{C_{n}^\Gamma(A\cap B)}\cong \frac{C_{n}^\Gamma(A)+C_{n}^\Gamma(B)}{C_{n}^\Gamma(A)}\cong \frac{C_{n}^\Gamma(G,\mathbb{U})}{C_{n}^\Gamma(A)}\cong \frac{C_{n}^\Gamma(G)}{C_{n}^\Gamma(A)}=C_{n}^\Gamma(G,A).\] It follows $H_{n}^\Gamma(G,A)\cong H_{n}^\Gamma(B,A\cap B)$.	
\end{proof}

\medskip

\begin{Def}
Let $G=(V,E)$ be the hypergraph, we define $k$-level closure and $k$-level interior for any subhypergraph $G'$, for $k\ge 0$, to be induced subhypergraph of $G$ whose vertex sets are given by, 

\begin{equation*} \begin{split}
    V(\cl^0(G'))=V', & ~~~~ V(\cl^{i}(G')) = V(\cl^{i-1}(\cl(G'))), ~~\text{for some positive integer } i,\\
V(\int^0(G')) =V', & ~~~~ V(\int^i(G'))=V(\int^{i-1}(\int(G'))), ~~ \text{for some positive integer } i.\end{split}
\end{equation*}

\medskip

\end{Def}
By using the above definitions, we have the following lemma.
\begin{lem}\label{iterative int and cl}
Let $G=(V, E)$ be the hypergraph and $G_1$ and $G_2$ be any subhypergraph of $G$. Then:
\begin{enumerate}
	\item ${\cl^{i}(G')} \subseteq {\cl^{i+1}(G')}$, $\int^{i}(G')\subseteq \int^{i-1}(\int(G'))$
	\item  $\int^i(G-G')=G-\cl^i(G')$, $G-\int^i(G')=\cl(G-G').$
	\item $G=\int^i(G_1)\cup \int^i(G_2)$ if and only if $\cl^i(G-G_2)\subseteq \int^i(G_1).$
\end{enumerate}
\end{lem}

\begin{proof}
The statements follow directly from the definitions by straightforward inductive arguments.
\end{proof}

\begin{prop}[Excision-like property]\label{gamma excion-like}

Let $G=(V,E)$ be hypergraph. Let $G_1=(V_1,E(V_1))$ and $G_2=(V_2,E(V_2))$ be subhypergraph with $\cl^i(G_1)\subseteq \int^i(G_2)$. Then the inclusion $(G-G_1,G_2-G_1)\rightarrow (G,G_2)$ induces isomorpohisms $H_{n}^\Gamma(G-G_1,G_2-G_1)\cong H_{n}^\Gamma(G,G_2)$ for all $n< 4i$. Equivalently, for interior cover $\{A,B\}$ of hypergraph $G=(V,E)$ the inclusion $(B,A\cap B)\rightarrow (G,A)$ induces isomorphisms $H_{n}^\Gamma(B,A\cap B)\cong H_{n}^\Gamma(G,A), \forall n<4i$.
\end{prop}
To prove the excision-like property for the $\Gamma$-homology, we first establish the necessary lemmas and definitions.
\begin{lem}\label{x-x'=2}
    Let $x,x'\in V_n^\Gamma$ with $|x-x'|=2$ then:
        $$x' \in V_{\operatorname{cl}(x)}.$$
\end{lem}

\begin{proof}
Let $x=(x_1,\ldots,x_n)$ and $x'=(x'_1,\ldots,x'_n)$. Define, 
\[R_i:V_n^\Gamma\rightarrow V_n^\Gamma ~~\text{by}~~ (x_1\ldots,x_n)\mapsto (x_1,\ldots,x_{i-1},1-x_i,x_{i+1},\ldots,x_n) .\]
Without loss of generality we can assume that $x'=R_i R_j(x),~ \text{where}~i,j\in \{1,\ldots,n\}$ and $i\ne j$. Let $y=R_i(x)$. It follows $|x-y|=1$ and $|y-x'|=1$. Furthermore, $x,x'\in e_y$ by definition of $J_n^\Gamma.$ Since $x'\in e_y$ and $e_y\cap \{x\}\neq \phi$, $x'\in V_{\cl(x)}.$
\end{proof}


\begin{Def}
Let $J_n^\Gamma$ be a $\Gamma$-hypercube and let $x \in V_n^\Gamma$.
For $m\ge 0$, define
\[
T_m(x)=\{t\in V_n^\Gamma : |t-x|\le m\}.
\]
\end{Def}
\begin{lem}\label{closure and T_m}
Let $J_{2n}^\Gamma=(V_{2n}^\Gamma,E_{2n}^\Gamma)$ denote the $\Gamma$-hypercube
and let $x\in V_n^\Gamma$. For $k\ge0$, the hypergraph $\cl^k(x)$ is the induced subhypergraph
of $J_{2n}^\Gamma$ whose vertex set is $T_{2k}(x)$. Therefore, for $x\in V_n^\Gamma$, $\cl^n(x)=J_{2n}^\Gamma.$
\end{lem}
    \begin{proof}
We argue by induction on $k$.
For $k=0$, Lemma~\ref{x-x'=2} implies that the vertex set of
$\cl^0(x)$ consists only of $x$. Hence
\[
V_{\cl^0(x)}=\{x\}=T_0(x).
\]
\\
Assume now that $V_{\cl^k(x)}=T_{2k}(x)$.
Since $\cl^{k+1}(x)=\cl(\cl^k(x))$, every vertex of $\cl^{k+1}(x)$
is either in $\cl^k(x)$ or is at distance $2$ from some vertex
$y\in \cl^k(x)$.
Let $z$ be such a vertex. Then $|y-z|=2$, and since
$y\in T_{2k}(x)$ we have $|y-x|\le 2k$.
By the triangle inequality,
\[
|z-x|\le |z-y|+|y-x|\le 2+2k=2k+2,
\]
so $z\in T_{2k+2}(x)$.
\\
Conversely, if $z\in T_{2k+2}(x)$, then $|z-x|\le 2k+2$.
If $|z-x|\le 2k$ then $z\in\cl^k(x)$ by the induction hypothesis.
Otherwise $|z-x|=2k+2$, and there exists a vertex $y$
with $|y-x|\le 2k$ and $|z-y|=2$.
Thus $y\in\cl^k(x)$ and Lemma~\ref{x-x'=2} implies that
$z\in\cl^{k+1}(x)$.
\\
Therefore, $V_{\cl^{k+1}(x)}=T_{2k+2}(x)$.
\end{proof}

\begin{lem}\label{cl of vertex}
Let $f: G\to G'$ be a hypergraph morphism and let $v$ be a vertex of $G$.
Then
\[
f(\cl(v))\subseteq \cl(f(v)).
\]
\end{lem}
\begin{proof}
Let $x\in f(\cl(v))$. Then there exists $y\in \cl(v)$ such that
$f(y)=x$.\\
By the definition of closure, there exists a hyperedge $e$ of $G$
such that $v,y\in e$.
Since $f$ is a hypergraph morphism, there exists a hyperedge
$e'$ of $G'$ with
\[
f(e)\subseteq e'.
\]
In particular, $f(v),f(y)\in e'$.
Thus $x=f(y)$ lies in a hyperedge of $G'$ containing $f(v)$,
and therefore $x\in \cl(f(v))$.
\end{proof}

\begin{lem}\label{gamma interior}
Let $G=(V,E)$ be a hypergraph and let $A$ be a subhypergraph. Let $x\in V_n^\Gamma$ and let
$\sigma:J_n^\Gamma\to G$ be a hypergraph morphism. If
$\sigma(x)\in V(\int^{i}(A))$, then
\[
\sigma(N(x,J_n^\Gamma))\subseteq V(\int^{i-1}(A)).
\]
\end{lem}

\begin{proof}
Let $y\in \sigma(N(x,J_n^\Gamma))$. Then there exists
$z\in N(x,J_n^\Gamma)$ such that $\sigma(z)=y$.
Since $\sigma$ is a hypergraph morphism, it follows that
$y\in N(\sigma(x),G)$.\\
Suppose, for contradiction, that $y\notin V(\int^{i-1}(A))$.
By the definition of interior,
\[
V(\int^{i-1}(A))=V(G)- V(\cl^{i-1}(G-A)),
\]
and hence $y\in V(\cl^{i-1}(G-A))$.
Since $y\in N(\sigma(x),G)$, this implies
$\sigma(x)\in V(\cl^{i}(G-A))$.
Therefore
\[
\sigma(x)\notin V(\int^{i}(A))
= V(G)- V(\cl^{i}(G-A)),
\]
which contradicts the assumption that
$\sigma(x)\in V(\int^{i}(A))$.\\
Thus $y\in V(\int^{i-1}(A))$, and therefore
\[
\sigma(N(x,J_n^\Gamma))\subseteq V(\int^{i-1}(A)).
\]
\end{proof}

\begin{lem}\label{gamma excision-like proof}
	Let $G=(V, E)$ be the hypergraph, $G_1$ and $G_2$ be any two subhypergraphs such that there is some positive integer $i$ with $G=\int^i(G_1)\cup \int^i(G_2)$. Then for $n\leq 4i$ and for every $(n,\Gamma)$-hypercube $\sigma: J_n^\Gamma \rightarrow G$, $Im(\sigma)\subset G_1$ or $Im(\sigma) \subset G_2.$
\end{lem}
\begin{proof} Let $\sigma$ be a $(n,\Gamma)$-hypercube and $x\in V^\Gamma.$ Note that by lemma $\ref{closure and T_m}$, $T_{4i}=V_n^\Gamma$. Therefore, we show that either $\sigma(T_{4i})\subseteq G_1$ or $\sigma(T_{4i})\subseteq G_2$. Without loss of generality let  $\sigma(T_0(x))\subset \int^i(G_1)$. Then from \ref{closure and T_m} and \ref{gamma interior},  
\[\sigma(T_2(x))\subseteq \int^{i-1}(G_1)\] 
and iteratively we have $\sigma(T_{2i}(x))\subseteq \int^{i-i}(G_1)=G_1$.\\ 
Now if $\sigma(T_{4i}-T_{2i})\subseteq G_1$, then we have the desired result. If $\sigma(T_{4i}-T_{2i})\not\subseteq G_1$ then \[\sigma(T_{2i})\subseteq G-G_1, \sigma(T_{2i+2})\subseteq \cl(G-G_1).\] Iteration over $i$ gives $\sigma(T_{2i+2i})\subseteq \cl^i(G-G_1).$ Now by lemma \ref{iterative int and cl}, 
\[\cl^i(G-G_1)\subseteq\int^i(G_2).\] Thus, $\sigma(T_{4i})\subseteq \int^i(G_2)$ and the result follows.
\end{proof}

\begin{proof}[\ref{gamma excion-like},excision-like]

Proof follows in the same pattern as the proof of \ref{gamma excision} by using [\ref{gamma excision-like proof}]

\end{proof}

\begin{lem}\label{box excision lemma}
Let $G=(V,E)$ be the hypergraph, $G_1$ and $G_2$ be any two subhypergraph with  $G=\int(G_1)\cup \int(G_2)$. Then for $n\leq 1$ and for every $(n,\Box)$-hyppercube $\sigma$, $Im(\sigma)\subset G_1$ or $Im(\sigma) \subset G_2.$ In particular, $C_n^\Box(G)=C_n^\Box(G,\mathbb{U})$, where $\mathbb{U}=\{G_1,G_2\}$ and $n\in\{0,1\}$.
\end{lem}

\begin{proof} Note that by definition, $C_n^\Box(G,\mathbb{U}) \subset C_n^\Box(G)$. For converse we proceed through induction on $n$\\

For $n=0$, let $\sigma \in  C_0^\Box(G)$. Clearly, either $Im(\sigma) \in \int(G_1)$ or $Im(\sigma) \in \int(G_2)$. It implies $C_0^\Box(G)\subseteq C_0^\Box(G,\mathbb{U})$.\\

For $n=1$, in this case, the set $Im(\sigma) \subset X$ comprises two elements, namely, $\sigma(0)$ and $\sigma(1)$. We can assume without loss of generality that the element $\sigma(0) \in \int(G_1)$. Then by definition of hypergraph morphism $\sigma(0)\in e$, for some hyperedge $e$ in $G_1.$ Note that, $\sigma(1)\in e$, since $\{0,1\}$ is a hyperedge in $J_1^\Box$ we get $Im(\sigma) \subset G_1$.
\end{proof}

Next we show that Lemma \ref{box excision lemma} fails for the case $n\ge 2$.

\begin{ex}
Let $G$ be the hypergraph with vertex set $V=\{v_1,v_2,v_3,v_4\}$ and set of hyperedges $E=\{\{v_1,v_2\} , \{v_2,v_3\}, \{v_3,v_4\}\}$. Let vertex set of $G_1$ and $G_2$ be $\{v_1,v_2,v_3\} $ and $ \{v_2,v_3,v_4\} $, respectively. Then note that $\int(G_1)=\{v_1,v_2\}$ and $\int(G_2)=\{v_3,v_4\}$ Define $\sigma: J_2^\Box \rightarrow G$ by $\sigma(0,0)=v_1$, $\sigma(1,0)=v_2$,  $\sigma(0,1)=v_3$ and $\sigma(1,1)=v_4$. Note that $Im(\sigma)\notin G_i$ for $i=\{1,2\}.$
\end{ex}

\begin{lem}\label{lemma box excision-like}
Let $G=(V,E)$ be the hypergraph, $G_1$ and $G_2$ be any two subhypergraphs such that there is some positive integer $i$ with  $G=\int^i(G_1)\cup \int^i(G_2)$. Then for $n\leq i+1$ and for every $(n,\Box)$-hyppercube $\sigma$, $Im(\sigma)\subset G_1$ or $Im(\sigma) \subset G_2.$ In particular, $C_n^\Box(G)=C_n^\Box(G,\mathbb{U})$, where $\mathbb{U}=\{G_1,G_2\}$ and $n\leq i+1.$
\end{lem}
\begin{proof}
Let $\sigma$ be a $(n,\Box)$-hypercube and fix $x=(0,0,\ldots,0) \in V_n^\Box$. For $n=0$,  $\sigma(T_0(x)) \subset   \int(G_1)$.  
Then, by the definition of hypergraph morphism,
\[
 {Im}(\sigma) \subset   \int(G_1).
\]
For $0 \le n \le i+1$, without loss of generality let
\[ \sigma(T_0(x)) \subset   \int^i(G_1). \]
Since elements of $T_{j-1}(x)$ and $T_{j+1}(x)$ are neighbors of $T_j(x)$ for $j=1,2,\ldots, n-1 $, we have
\[
\sigma(T_j(x)) \subset   \int^{i-j}(G_1), \text{ by definition of $k$-level interior}.\]
Now, for $n=i+1,$ $\sigma(T_n(x)-T_{n-1}(x))$ may or may not be in the vertex set of $G_1$; if it is, we have the desired result, and if it is not, then we have $\sigma(T_n(x)-T_{n-1}(x))\subset G-G_1$, and we claim that $Im(\sigma)\subset G_2$. Since we know that from \ref{iterative int and cl}, \[G-G_1\subset \cl(G-G_1)\subset \cl^i(G-G_1)\subseteq \int^i(G_2).\] It implies $\sigma(T_n(x)-T_{n-1}(x))\subset \int^i(G_2)$ and using the definition of $k$-level interior and for  $j=2,\ldots,n$ we have $\sigma(T_{n}(x)-T_{n-j-1}(x))\subseteq \int^{i-j+1}(G_2)$. Thus, $Im(\sigma)\subset G_2$.
\end{proof}

\begin{thm}
(Excision-like property)\\
Let $G=(V,E)$ be hypergraph. Let $G_1=(V_1,E(V_1))$ and $G_2=(V_2,E(V_2))$ be subhypergraph such that $G_1\subseteq G_2$ and $cl^i(G_1)\subseteq  \int^i(G_2)$. Then the inclusion $(G-G_1,G_2-G_1)\rightarrow (G,G_2)$ induces isomorpohisms $H_{n}^\Box(G-G_1,G_2-G_1)\cong H_{n}^\Box(G,G_2)$ for all $n< i+1$. Equivalently, for interior cover $\{A,B\}$ of hypergraph $G=(V,E)$ the inclusion $(B,A\cap B)\rightarrow (G,A)$ induces isomorphisms $H_{n}^\Box(B,A\cap B)\cong H_{n}^\Box(G,A), \forall n<i+1$.
\end{thm}

\begin{proof}
The proof follows from Lemma \ref{lemma box excision-like}
\end{proof}

            \medskip

\newpage
\bibliography{ref}
\bibliographystyle{amsalpha}
\end{document}